\documentclass[11pt]{amsart}
\usepackage{amsmath}
\usepackage{amsfonts}
\usepackage{amssymb}

\newtheorem{teo}{Theorem}
\newtheorem{lem}{Lemma}[section]
\newtheorem{cor}[lem]{Corollary}
\newtheorem{prop}[lem]{Proposition}

\newtheorem{defi}[lem]{Definition}
\newtheorem{oss}[lem]{Remark}

\newcommand{\spa}{\mathrm{span}}
\newcommand{\dive}{\mathrm{div}}
\newcommand{\ud}{\mathrm{d}}

\renewcommand{\phi}{\varphi}
\newcommand{\R}{{\mathbb{R}}}

\newcommand{\bbR}{{\mathbb{R}}}

\newcommand{\bbS}{{\mathbb{S}}}

\newcommand{\average}{{\mathchoice {\kern1ex\vcenter{\hrule height.4pt
width 6pt
depth0pt} \kern-9.7pt} {\kern1ex\vcenter{\hrule height.4pt width 4.3pt
depth0pt}
\kern-7pt} {} {} }}

\newcommand{\LL}{\mathcal{L}}

\title[Symmetry results for nonlinear elliptic operators]{Symmetry results for nonlinear elliptic operators with unbounded drift}
\thanks{A.F. and M.N. are supported by the ERC grant  {\it EPSILON -- Elliptic Pde's and Symmetry of Interfaces and Layers for Odd Nonlinearities}. 
M.N. and A.P. acknowledge partial support by the CaRiPaRo project {\it Nonlinear Partial Differential Equations: models, analysis, and control-theoretic
problems}.}
\author{Alberto Farina, Matteo Novaga, Andrea Pinamonti
}
\address{LAMFA-CNRS UMR 7352, Universit\'e de Picardie Jules Verne, Facult\'e des Sciences, 33, Rue Saint-Leu, 80039,  Amiens, France}
\address {Institut Camille Jordan, CNRS UMR 5208, Universit\'e Claude Bernard, Lyon I, Villeurbanne, France}
\email{alberto.farina@u-picardie.fr}
\address{Dipartimento di Matematica\\Universit\`a di Padova\\ Via Trieste 63\\ Padova\\ Italy}
\email{pinamonti@science.unitn.it}
\address{Dipartimento di Matematica\\Universit\`a di Pisa\\ Largo Bruno Pontecorvo 5\\ Pisa\\ Italy}
\email{novaga@dm.unipi.it}
\begin{document}

\begin{abstract}
We prove the one-dimensional symmetry of solutions to 
elliptic equations of the form 
$-\dive(e^{G(x)}a(|\nabla u|)\nabla u)=f(u) e^{G(x)}$, under suitable 
energy conditions.
Our results hold without any restriction on the dimension of the ambient space.
\end{abstract}

\maketitle

\tableofcontents

\section{Introduction}  

In this paper we study the one-dimensional symmetry of solutions to nonlinear equations of the following type:
\begin{align}\label{eqn}
\dive(a(|\nabla u|)\nabla u)+a(|\nabla u|)\left\langle \nabla G(x), \nabla u\right\rangle +f(u)=0,
\end{align} 
or in a more compact form
\begin{align}\label{eqncomp}
-\dive(e^{G(x)}a(|\nabla u|)\nabla u)=f(u) e^{G(x)},
\end{align}
where $f \in C^{1}(\bbR)\footnote{One could consider functions $f$ which are only locally lipschitz continuous, as in  \cite{FSV}. To avoid inessential 
technicalities, we do not treat this case here.}$, $G\in C^{2}(\bbR^n)$ and 
$a\in C^{1,1}_{loc}((0,+\infty))$. We also require that the function $a$ satisfies 
the following structural conditions:
\begin{align}
&a(t)>0\quad \mbox{for any}\ t\in (0,+\infty),\\
&a(t)+a'(t)t>0\quad \mbox{for any}\ t\in (0,+\infty).
\end{align}

Observe that the general form of \eqref{eqncomp} encompasses, as very special cases, many elliptic singular and degenerate equations. Indeed, if $G\equiv 0$ and $a(t)=t^{p-2}$, $1<p<+\infty$, or $a(t)=1/ \sqrt{1+t^2}$ then we obtain the $p$-Laplacian and the mean curvature equations respectively. Moreover, if $a(t)\equiv 1$ and $G(x)=-|x|^2/2$ equation \eqref{eqn} boils down to the classical Ornstein-Uhlenbeck operator for which we refer to \cite{Bog} and the references therein.

To prove the one-dimensional symmetry of solutions we follow the approach introduced in \cite{FarHab} and further developed in \cite{FSV}.  
Following \cite{FarHab, FSV, CNP}, we define $A:\bbR^n\to Mat(n\times n), \lambda_1\in C^0((0,+\infty)),\lambda_G\in C^0(\bbR^{2n})$ as follow
\begin{align}
A_{hk}(\xi):=\frac{a'(|\xi|)}{|\xi|}\xi_h\xi_k+a(|\xi|)\delta_{hk}\quad \mbox{for any}\  1\leq h,k\leq n,
\end{align}
\begin{align}\label{defa}
\lambda_1(t):=a(t)+a'(t)t\quad \mbox{for any}\  t>0
\end{align}
and 
\begin{align}\label{deflambdaG}
\lambda_G(x):=\mbox{maximal eigenvalue of}\ \nabla^2G(x).	
\end{align}
%

\begin{defi} We say that $u$ is a weak solution to (\ref{eqn}) if $u\in C^1(\R^n)$, 
\begin{align}\label{weak}
\int_{\bbR^n}\left\langle a(|\nabla u|)\nabla u, \nabla\varphi\right\rangle-f(u)\varphi\  \ud \mu=0\qquad \forall \varphi\in C^1_c(\bbR^n)
\end{align}
and either (A1) or (A2) is satisfied, where :
\begin{itemize}
\item[(A1)] $\{\nabla u=0\}=\emptyset$. 

\item[(A2)]  $a\in C^0([0,+\infty))$
and
\[
\mbox{the map}\quad t\to ta(t)\quad \mbox{belongs to}\quad C^1([0,+\infty)).
\]
\end{itemize}
\end{defi}
 
\medskip
Notice that \eqref{weak} is well-defined, thanks to  (A1) or (A2).

\smallskip

Notice also that weak solutions to (\ref{eqn}) are critical points of the functional
\begin{align}\label{func}
I(u):=\int_{\bbR^n}\Big(\Lambda(|\nabla u|)+F(u)\Big)\ud \mu
\end{align}
where $F'(t)=-f(t)$, $\ud\mu=e^{G(x)}\ud x$ and
\[
\Lambda(t):=\int_{0}^t a(|\tau|)\tau \ud \tau.
\]

The regularity assumption $u\in C^1(\bbR^n)$ is always fulfilled in many important cases, like those involving the  $p$-Laplacian operator or the mean curvature operator. For instance, when $a(t)=t^{p-2}$, $1<p<+\infty$, any distribution solution $ u \in W^{1,p}_{loc}(\bbR^n) \cap L^{\infty}_{loc}(\bbR^n)$ is of class $C^1$, by the well-known results in \cite{LU68, T84}).  In light of this, and in view of the great generality of the function $a$, it is natural to work in the above setting.

\begin{defi}\label{phstab}
Let $h\in L^1_{loc}(\bbR^n)$ and let $u$ be a weak solution to \eqref{eqn}. We say that $u$ is $h-$stable if
\begin{align}\label{hstab}
\int_{\bbR^n}\left\langle A(\nabla u)\nabla\varphi,\nabla\varphi\right\rangle-f'(u)\varphi^2\ \ud\mu\geq \int_{\bbR^n}a(|\nabla u|)h\varphi^2\ \ud\mu\quad \forall \varphi\in C^{1}_c(\bbR^n).
\end{align}
\end{defi} 



\begin{oss}\rm 
 When $a(t)\equiv 1$, Definition \ref{phstab} boils down to the $h-$stability condition introduced in \cite{CNV,CNP}.

When $h\equiv 0$, then $u$ satisfies the classical stability condition \cite{FarHab, FSV, FV12, ASV2}, and 
we simply say that $u$ is stable. In particular, 
every minimum point of the functional \eqref{func} is a stable solution to \eqref{eqn}.

Let us also point out that, in view of (A1) or (A2),  the integral
\begin{align}
\int_{\bbR^n}\left\langle A(\nabla u)\nabla\varphi,\nabla\varphi\right\rangle-f'(u)\varphi^2- a(|\nabla u|)h\varphi^2\ \ud\mu
\end{align}
is well defined.\footnote{$\,$ cfr. also \cite[footnote 1 at p. 742 and footnote 2 at page 743]{FSV}.} In particular, under the condition (A2) the function $A$ can be extended by continuity at the origin, by setting $ A_{hk}(0):= a(0) \delta_{hk}$.


\end{oss}

\smallskip

We can now state our main symmetry results:

\smallskip

\begin{teo}\label{1Dgen}
Assume $G\in C^2(\bbR^n)$ and $h\in L^1_{loc}(\bbR^n)$ with $h\geq \lambda_G$. Let $u\in C^1(\bbR^n)\cap C^{2}(\{\nabla u\neq 0\})$ with $\nabla u\in H^1_{loc}(\bbR^n)$ be a $h-$stable weak solution to \eqref{eqn}. 
Assume that there exists $C>0$ such that 
\begin{align}\label{stimaaut}
	\lambda_1(t)\leq C a(t)\quad \forall t>0,
\end{align}
and one of the following conditions hold
\begin{enumerate}
	\item[(a)] there exists $C_0\geq 1$ such that $\int_{B_R}a(|\nabla u|)|\nabla u|^2\ud\mu\leq C_0 R^2$ for any $R\geq C_0$,
	\item[(b)] $n=2$ and $u$ satisfies $a(|\nabla u|)|\nabla u|^2e^{G}\in L^{\infty}(\bbR^2)$.
\end{enumerate}
Then $u$ is one-dimensional, i.e. there exists $\omega\in \bbS^{n-1}$ and $u_0:\bbR\to\bbR$ such that 
\begin{align}
u(x)=u_0(\left\langle \omega,x\right\rangle)\quad \forall x\in\bbR^n.
\end{align}
Moreover, 
\begin{align}\label{ugu1}
	\left\langle(h(x) {\rm I}_n-\nabla^2 G(x))\nabla u, \nabla u\right\rangle=0\quad \forall x\in\bbR^n.
\end{align}
In particular, if $u_0$ is not constant, there are $C$ and $g$ of class $C^2$ such that
\begin{align}\label{ugu22}
	G(x)=C(\left\langle \omega,x\right\rangle)+g(x'),
\end{align}
where $x':=x-\left\langle \omega,x\right\rangle\omega$ and $\lambda_G(x)=h(x)=C''(\left\langle \omega,x\right\rangle)$ for all $x\in\bbR^n$.
\end{teo}

\smallskip

\begin{oss}\rm 
Paradigmatic examples satisfying the assumption \eqref{stimaaut} are the $p$-Laplacian operator, for any $ p \in (1,+\infty)$, and the generalized mean curvature operator obtained  by setting $a(t):=(1+t^q)^{-\frac{1}{q}}$, with $q>1$. 
\end{oss}

\smallskip

\begin{teo}\label{propex}
Let $G(x):=-|x|^2/2$, $a(t):=t^{p-2}$ with $p > 1$ and let $u\in C^1(\bbR^n)\cap W^{1,\infty}(\bbR^n)$ be a monotone weak solution to \eqref{eqn}, i.e., such that 
\begin{align}
\partial_i u(x) >0 \quad \forall x\in\bbR^n,
\end{align}
for some $i\in \{1,\ldots, n\}.$

\noindent Suppose that $u$ satisfies either (a) or (b) in Theorem \ref{1Dgen}. Then $u$ is one-dimensional. Moreover, if either $p=2$ or $a(t):=(1+t^q)^{-\frac{1}{q}}$ with $q>1$, then the same conclusion holds for every monotone weak solution $u\in C^1(\bbR^n)\cap L^{\infty}(\bbR^n)$. 
\end{teo}

\smallskip

\begin{teo}\label{teomorse}
Let $u$ be a bounded 
weak solution to 
\begin{align}
	\Delta u-\left\langle x,\nabla u\right\rangle+f(u)=0
\end{align}
with Morse index $k$. Then, 
\begin{enumerate}
	\item[(i)] if $k\le 2$ then $u$ is one-dimensional;
	\item[(ii)] if $3\leq k \leq n$ then $u$ is a function of at most $k-1$ variables, i.e. there exists $C\in Mat((k-1)\times n)$ and $u_0:\bbR^{k-1}\to\bbR$ such that
\begin{align}\label{kdim}
	u(x)=u_0(Cx)\quad \forall x\in\bbR^n.
\end{align}
\end{enumerate}
\end{teo}

\section{A geometric Poincar\'e inequality}

We start by recalling the following Lemma which has been proved in \cite{FSV}.
\begin{lem}
For any $\xi\in \bbR^n\setminus \{0\}$, the matrix $A(\xi)$ is symmetric and positive definite and its eigenvalues are 
$\lambda_1(|\xi|),\cdots, \lambda_n(|\xi|)$, where $\lambda_1$ is as in \eqref{defa} and $\lambda_i(t)=a(t)$ for every $i=2,\ldots, n$. Moreover,
\begin{align}
\left\langle A(\xi)\xi,\xi\right\rangle=|\xi|^2\lambda_1(|\xi|),
\end{align}
and
\begin{align}\label{posdef}
0\leq \left\langle A(\xi)(V-W),(V-W)\right\rangle=\left\langle A(\xi)V,V\right\rangle+\left\langle A(\xi)W,W\right\rangle-2\left\langle A(\xi)V,W \right\rangle,
\end{align}
for any $V,W\in\bbR^n$ and any $\xi \in \bbR^n\setminus\{0\}$.
\end{lem}
\begin{lem}\label{eqnlin}
Let $u\in C^1(\bbR^n)\cap C^2(\{\nabla u\neq 0\})$ with $\nabla u\in H^1_{loc}(\bbR^n)$ be a weak solution to (\ref{eqn}). 
Then for any $i=1,\ldots, n$, and any $\varphi\in C^1_c(\R^n)$ we have
\begin{align}\label{eqngen}
\int_{\bbR^n} \left\langle A(\nabla u)\nabla u_i, \nabla\varphi \right\rangle-a(|\nabla u|)\left\langle \nabla u,\nabla(G_i)\right\rangle\varphi-f'(u)u_i\varphi\  \ud\mu =0.
\end{align}
\end{lem}
\begin{proof}
By Lemma $2.2$ in \cite{FSV} we have
\begin{align}
\mbox{the map}\quad x\to W(x):=a(|\nabla u(x)|)\nabla u(x)\quad \mbox{belongs to}\ W^{1,1}_{loc}(\bbR^n,\bbR^n),
\end{align}
therefore, since $e^{G(x)}\in C^{2}(\bbR^n)$ we get  
\begin{align}\label{wq}
We^{G}\in W^{1,1}_{loc}(\bbR^n,\bbR^n).
\end{align}
By Stampacchia's Theorem (see, e.g. \cite[Theorem 6:19]{LL}), we get $\partial_i(We^{G})=0$ for almost any $x\in \{We^G=0\}=\{W=0\}$, that is
\[
\partial_i(We^{G})=0
\]
for almost any $x\in \{\nabla u=0\}$. In the same way, by Stampacchia's Theorem and (A2), it can be proven that $\nabla u_i(x)=0$, and hence $A(\nabla u(x))\nabla u_i(x)=0$, for almost any $x\in \{\nabla u=0\}$.
Moreover, the following relation holds (see \cite{FSV} for the proof)
\begin{align}
\partial_i(W e^{G})=(A(\nabla u)\nabla u_i + a(|\nabla u|)\nabla uG_i)e^{G} \quad\ \mbox{on}\ \{\nabla u\neq 0\},
\end{align}
and thanks to the previous observations
\begin{align}\label{wq2}
	\partial_i(W e^{G})=(A(\nabla u)\nabla u_i + a(|\nabla u|)\nabla uG_i)e^{G} \quad a.e.\ \mbox{in}\ \bbR^n.
\end{align}
Applying (\ref{weak}) with $\varphi$ replaced by $\varphi_i$ and making use of \eqref{wq} and \eqref{wq2}, we obtain
\begin{align*}
0&=\int_{\bbR^n}a(|\nabla u|)\left\langle \nabla u,\nabla\varphi_i\right\rangle+f(u)\varphi_i\ \ud\mu\\
&=-\int_{\bbR^n}\left\langle A(\nabla u)\nabla u_i,\nabla\varphi\right\rangle+a(|\nabla u|)\left\langle \nabla u,\nabla \varphi\right\rangle G_i\ \ud\mu\\
&-\int_{\bbR^n}f'(u)u_i\varphi+f(u)\varphi G_i\ \ud\mu\\
&=-\int_{\bbR^n}\left\langle A(\nabla u)\nabla u_i,\nabla\varphi\right\rangle+a(|\nabla u|)\left\langle \nabla u,\nabla( \varphi G_i)\right\rangle \ud\mu\\
&-\int_{\bbR^n}-a(|\nabla u|)\left\langle \nabla u, \nabla G_i\right\rangle \varphi+f'(u)u_i\varphi+f(u)\varphi G_i\ \ud\mu.
\end{align*}
Recalling (\ref{weak}), applied with $\varphi$ replaced by $\varphi G_i$, we obtain the thesis. 
\end{proof}

>From now on,  we use $A$ and $a$,  as a short-hand notation for $A(\nabla u)$ and $a:=a(|\nabla u|)$ respectively. 

In the following result we prove that every monotone solution to \eqref{eqn} is indeed $h-$stable.
\begin{lem}\label{stab}
Assume that $u$ is a weak solution to \eqref{eqn} and that there exists $i\in \{1,\ldots, n\}$ such that
\begin{align}\label{monot}
u_i:= \partial_i u(x) >0 \quad \forall x\in\bbR^n
\end{align}
then $u$ is $h-$stable, with 
\[
h(x):=\frac{\left\langle \nabla u(x), \nabla G_i(x)\right\rangle}{u_i(x)}
\]
\end{lem}
\begin{proof}
Let $\phi\in C^{\infty}_c(\bbR^n)$ and $\psi:=\phi^2/u_i$. 
We use \eqref{posdef} with $V:=\phi\nabla u_i/u_i$ and $W:=\nabla \phi$ to obtain that
\[
\frac{2\phi}{u_i}\left\langle A\nabla u_i,\nabla\phi\right\rangle-\frac{\phi^2}{u_i^2}\left\langle A\nabla u_i,\nabla u_i\right\rangle\leq \left\langle A\nabla \phi,\nabla\phi\right\rangle.
\]
>From this and Lemma \ref{eqnlin} we get
\begin{align}
0&=\int \left\langle A\nabla u_i,\nabla \psi\right\rangle-a\left\langle \nabla u,\nabla G_i\right\rangle\psi-f'(u)u_i\psi\ \ud\mu\\
\nonumber
&=\int 2\frac{\phi}{u_i}\left\langle A\nabla u_i,\nabla \phi\right\rangle-\frac{\phi^2}{u_i^2}\left\langle A\nabla u_i,\nabla u_i\right\rangle-a\frac{\phi^2}{u_i}\left\langle \nabla u,\nabla G_i\right\rangle-f'(u)\phi^2\ \ud\mu\\
\nonumber
&\leq \int \left\langle A\nabla\phi,\nabla\phi\right\rangle-a\frac{\phi^2}{u_i}\left\langle \nabla u,\nabla G_i\right\rangle-f'(u)\phi^2\ \ud\mu.
\end{align}
Notice that we can apply Lemma \ref{eqnlin} since, in view of  \eqref{monot}, $u$ has no critical points and thus it is of class $C^2$, by the classical regularity results. 
\end{proof}
The following Lemma can be proved using the same tecniques implemented in \cite[Lemma 2.4]{FSV}, 
\begin{lem} \label{approtest}
Let $h\in L^1_{loc}(\bbR^n)$. Let $u\in C^1(\bbR^n)\cap C^2(\{\nabla u\neq 0\})$ with $\nabla u\in H^1_{loc}(\bbR^n)$ be a $h-$stable weak solution to \eqref{eqn}. Then, \eqref{hstab} holds for any $\varphi\in H^1_0(B)$ and for any ball $B\subset \bbR^n$. Moreover, under the assumptions of Lemma \ref{eqnlin},
	\begin{align}\label{eqngen23}
\int_{\bbR^n} \left\langle A(\nabla u)\nabla u_i, \nabla\varphi \right\rangle-a(|\nabla u|)\left\langle \nabla u,\nabla(G_i)\right\rangle\varphi-f'(u)u_i\varphi\  \ud\mu =0.
\end{align}
for any $i=1,\ldots,n$, any $\varphi\in H^1_0(B)$ and any ball $B\subset\bbR^n$.
\end{lem}
\begin{prop}
Let $h\in L^1_{loc}(\bbR^n)$ and $u\in C^1(\bbR^n)\cap C^2(\{\nabla u\neq 0\})$ with $\nabla u\in H^1_{loc}(\bbR^n)$ be a $h-$stable weak solution to (\ref{eqn}). 
Then, for every $\varphi\in C^1_c(\bbR^n)$ it holds
\begin{align}
 \int_{\bbR^n}a(|\nabla u|) h(x)|\nabla u|^2\varphi^2\ \ud\mu&\leq \int_{\bbR^n}|\nabla u|^2\left\langle A\nabla\varphi,\nabla\varphi\right\rangle+a(|\nabla u|)\left\langle\nabla^2 G \nabla u,\nabla u\right\rangle\varphi^2\\
\nonumber
&+\varphi^2\Big[\left\langle A\nabla|\nabla u|,\nabla|\nabla u|\right\rangle-\sum_{i=1}^n\left\langle A(\nabla u)\nabla u_i, \nabla u_i \right\rangle\Big]\ud\mu.
\end{align}
\end{prop}


\begin{proof}
We start observing that by Stampacchia's Theorem, since $\mu<<\LL^{n}$, we get
\begin{align}
&\nabla|\nabla u|(x)=0\quad \mu-\mbox{a.e.}\ x\in \{|\nabla u|=0\}\\
&\nabla u_j(x)=0\quad \mu-\mbox{a.e.}\ x\in \{|\nabla u|=0\}\subseteq \{u_j=0\},
\end{align}
for any $j=1,\ldots, n$.
Let $\varphi\in C^1_c(\bbR^n)$ and $i=1,\ldots,n$. Using (\ref{eqngen}) with test function $u_i\varphi^2$ and summing over $i=1,\ldots, n$ we get
\begin{align}\label{we4}
\int_{\bbR^n}\sum_{i=1}^n \left\langle A(\nabla u)\nabla u_i, \nabla(u_i\varphi^2) \right\rangle -f'(u)|\nabla u|^2\varphi^2\  \ud\mu =\int_{\bbR^n}a(|\nabla u|)\left\langle \nabla^2 G\nabla u,\nabla u\right\rangle \varphi^2\ \ud\mu
\end{align}
Using (\ref{hstab}) with test function $|\nabla u|\varphi$ (note that this choice is possible thanks to Lemma \ref{approtest}) we then get
\begin{align}
 \int_{\bbR^n}a(|\nabla u|)h(x)|\nabla u|^2\varphi^2\ \ud\mu&\leq \int_{\bbR^n}\left\langle\Big(A(\nabla u(x))\nabla(|\nabla u|\varphi)\Big),\nabla(|\nabla u|\varphi)\right\rangle-f'(u)|\nabla u|^2\varphi^2\ \ud\mu\\
\nonumber
&=\int_{\bbR^n}|\nabla u|^2\left\langle A\nabla\varphi,\nabla\varphi\right\rangle\ud\mu+\int_{\{\nabla u\neq 0\}}\varphi^2\left\langle A\nabla|\nabla u|,\nabla|\nabla u|\right\rangle\\
\nonumber
&+2\varphi|\nabla u|\left\langle A\nabla\varphi,\nabla|\nabla u|\right\rangle-f'(u)|\nabla u|^2\varphi^2\ \ud\mu
\end{align}
and by (\ref{we4}) we conclude that
\begin{align}\label{prepoi}
 \int_{\bbR^n}a(|\nabla u|)h(x)|\nabla u|^2\varphi^2\ \ud\mu&\leq \int_{\bbR^n}|\nabla u|^2\left\langle A\nabla\varphi,\nabla\varphi\right\rangle\ud\mu+\int_{\{\nabla u\neq 0\}}a(|\nabla u|)\left\langle\nabla^2 G \nabla u,\nabla u\right\rangle\varphi^2\ud\mu\\
\nonumber
&+\int_{\{\nabla u\neq 0\}}\varphi^2\Big[\left\langle A\nabla|\nabla u|,\nabla|\nabla u|\right\rangle-\sum_{i=1}^n\left\langle A(\nabla u)\nabla u_i, \nabla u_i \right\rangle\Big]\ud\mu.
\end{align}
which is the thesis.
\end{proof}

\begin{oss}\rm \label{rempoi}
Letting
\[
L_{u,x}:=\{y\in\bbR^n\ |\ u(y)=u(x)\}\,,
\]
we denote by $\nabla_{T}u$ the tangential gradient of $u$ along $L_{u,x}\cap \{\nabla u\neq 0\}$,
and by $k_1,\ldots, k_{n-1}$ the principal curvatures of $L_{u,x}\cap \{\nabla u\neq 0\}$.
By Lemma $2.3$ in \cite{FSV} we obtain  
\begin{align}\label{ugu}
&\left\langle A\nabla|\nabla u|,\nabla|\nabla u|\right\rangle-\sum_{i=1}^n\left\langle A(\nabla u)\nabla u_i, \nabla u_i \right\rangle=a \Big[|\nabla |\nabla u||^2-\sum_{i=1}^n|\nabla u_i|^2\Big]-a'|\nabla u||\nabla_T|\nabla u||^2
\end{align}
and using \eqref{defa} we get
\begin{align}\label{ugu2}
&\left\langle A\nabla|\nabla u|,\nabla|\nabla u|\right\rangle-\sum_{i=1}^n\left\langle A(\nabla u)\nabla u_i, \nabla u_i \right\rangle\\
\nonumber
&=-\lambda_1|\nabla_T|\nabla u||^2-a(|\nabla u|)\Big(\sum_{i=1}^n|\nabla u_i|^2-|\nabla_T|\nabla u||^2-|\nabla|\nabla u||^2\Big)
\end{align}
Notice that
the quantity $$\sum_{i=1}^n|\nabla u_i|^2-|\nabla |\nabla u||^2-|\nabla_T|\nabla u||^2 $$  has a geometric interpretation, 
in the sense that  it can be expressed in terms of the principal curvatures of level sets of $u$.
More precisely, the following formula holds (see \cite{FSV,SZ1,SZ2})
\begin{align}\label{eqfund}
	\sum_{i=1}^n|\nabla u_i|^2-|\nabla |\nabla u||^2-|\nabla_T|\nabla u||^2=|\nabla u|^2 \sum_{j=1}^{n-1} k_j^2
	\qquad {\rm on\ }L_{u,x}\cap \{\nabla u\neq 0\}\,,
\end{align}
so that \eqref{prepoi} becomes
\begin{align*}
&\int_{\{\nabla u\neq 0\}}a(|\nabla u|)h(x)|\nabla u|^2\varphi^2+\Big[\lambda_1|\nabla_T|\nabla u||^2+a(|\nabla u|)|\nabla u|^2 \sum_{j=1}^{n-1} k_j^2\Big]\varphi^2
\\
& \quad -a(|\nabla u|)\left\langle\nabla^2 G \nabla u,\nabla u\right\rangle\varphi^2\ \ud\mu
\\
&\leq \int_{\bbR^n}\left\langle A\nabla\varphi,\nabla\varphi\right\rangle|\nabla u|^2\ud\mu.
\end{align*}

Rearranging the terms, we obtain

\begin{align}\label{prepoi2}\nonumber
 &\int_{\{\nabla u\neq 0\}}a(|\nabla u|)\left\langle (h(x)I-\nabla^2G)\nabla u,\nabla u\right\rangle\varphi^2+\Big[\lambda_1|\nabla_T|\nabla u||^2+a(|\nabla u|)|\nabla u|^2 \sum_{j=1}^{n-1} k_j^2\Big]\varphi^2\ \ud\mu\\
 &\leq \int_{\bbR^n}\left\langle A\nabla\varphi,\nabla\varphi\right\rangle|\nabla u|^2\ud\mu,
\end{align}
where $I\in Mat(n\times n)$ denotes the identity matrix.\\
Notice that from \eqref{prepoi2} we also obtain 
\begin{align}\label{senza}
\int_{\{\nabla u\neq 0\}}a(|\nabla u|)\left\langle (h(x)I-\nabla^2G)\nabla u,\nabla u\right\rangle\varphi^2\ud\mu
 \leq \int_{\bbR^n}\left\langle A\nabla\varphi,\nabla\varphi\right\rangle|\nabla u|^2\ud\mu.
\end{align}
\end{oss} 

\section{One-dimensional symmetry of solutions}
In this section we will use \eqref{prepoi2} to prove several one-dimensional results for solutions to \eqref{eqn}, following the approach introduced in \cite{FarHab} and then developed in \cite{FSV}. Notice that, more recently,  a similar approach has also been used to handle semilinear equations in riemannian and subriemannian spaces (see \cite{FMV, fsv1,fsv2,FP,FV1, PV}) and also to study problems involving the Ornstein-Uhlenbeck operator \cite{CNV}, as well as semilinear equations with unbounded drift \cite{CNP}. 

\vspace{5pt}

The following Lemma is proved in \cite{FSV, FV1}.
\begin{lem}\label{lemfond}
Let $g\in L^{\infty}_{loc}(\bbR^n, [0,+\infty))$ and let $q>0$. Let also, for any $\tau>0$,
\begin{align}
\eta(\tau):=\int_{B_\tau} g(x)\ud x.
\end{align}
Then, for any $0<r<R$,
\begin{align}
\int_{B_R\setminus B_r} \frac{g(x)}{|x|^q}\ud x\leq q\int_r^R\frac{\eta(\tau)}{|\tau|^{q+1}}\ud\tau+\frac{1}{R^q}\eta(R)
\end{align}
\end{lem}

\smallskip

\noindent{\it Proof of Theorem \ref{1Dgen}.}
Let us fix $R>0$ (to be taken appropriately large in what follows) and $x\in\bbR^n$ and let us define 
\begin{equation}\label{test}
\varphi(x):=\left\{\begin{array}{lll}
1 & \mbox{if}& x\in B_{\sqrt{R}}\\
2\frac{\log(R/|x|)}{\log(R)} & \mbox{if}& x\in B_R\setminus B_{\sqrt{R}}\\
0& \mbox{if}& x\in \bbR^n\setminus B_R,
\end{array}
\right.
\end{equation}
where $B_R:=\{y\in\bbR^n\ |\ |y|< R\}$.
Obviously $\varphi\in Lip(\bbR^n)$ and 
$$|\nabla \varphi (x)|\leq C_2\frac{\chi_{\sqrt{R},R}(x)}{\log(R)|x|}$$ 
for suitable $C_2>0$.
Hence for every $R>e$, \eqref{prepoi2} together with $h\geq \lambda_G$ yields
\begin{align}\label{xineq2}
	\int_{\{\nabla u\neq 0\}\cap \overline{B}_R}\Big[\lambda_1|\nabla_T|\nabla u||^2+a(|\nabla u|)|\nabla u|^2 \sum_{j=1}^{n-1} k_j^2\Big]\varphi^2\ \ud\mu \leq \int_{\bbR^n}\left\langle A(\nabla u)\nabla\varphi,\nabla\varphi\right\rangle|\nabla u|^2\ud\mu
\end{align}
therefore, by \eqref{stimaaut}
\begin{align}\label{xineq3}
	\int_{\{\nabla u\neq 0\}\cap \overline{B}_R}\Big[\lambda_1|\nabla_T|\nabla u||^2+a(|\nabla u|)|\nabla u|^2 \sum_{j=1}^{n-1} k_j^2\Big]\varphi^2\ \ud\mu &\leq (1+C) \int_{\bbR^n}a(|\nabla u|)|\nabla \varphi|^2|\nabla u|^2\ud\mu\\
	\nonumber
	&\leq \frac{(1+C)C_2^2}{\log(R)^2}\int_{B_R\setminus B_{\sqrt{R}}}\frac{a(|\nabla u|)|\nabla u|^2}{|x|^2}\ud\mu
\end{align}
Applying Lemma \ref{lemfond} with $g=a(|\nabla u|)|\nabla u|^2 e^G$ and $q=2$, and recalling that 
\[
\int_{B_R}a(|\nabla u|)|\nabla u|^2\ud\mu\leq C_0 R^2
\]
for $R$ large, we obtain 
\begin{align}\label{xineq4}
	\int_{\{\nabla u\neq 0\}\cap \overline{B}_R}\Big[\lambda_1|\nabla_T|\nabla u||^2+a(|\nabla u|)|\nabla u|^2 \sum_{j=1}^{n-1} k_j^2\Big]\varphi^2\ \ud\mu&\leq \frac{(1+C) C_0C_2^2}{\log(R)^2}\Big[2\int_{\sqrt{R}}^R\frac{1}{|\tau|}\ud\tau+1\Big]\\
	\nonumber
	&\leq 2\frac{(1+C) C_0C_2^2}{\log(R)}.
\end{align}
Therefore, sending $R\to +\infty$ in \eqref{xineq4} we get
\begin{align}\label{trentuno}
	k_j(x)=0\quad \mbox{and}\quad |\nabla_T|\nabla u||(x)=0
\end{align}
for every $j=1,\ldots, n-1$ and every $x\in\{\nabla u\neq 0\}$. From this and Lemma $2.11$ in \cite{FSV} we get the one-dimensional symmetry of $u$.

Let us now suppose $n=2$ and $a(|\nabla u|)|\nabla u|^2e^{G}\in L^\infty(\R^2)$. 
Taking in \eqref{prepoi2} the following test function   
\begin{equation}\label{test2} \phi(x)=\max\left[0, \min\left(1, \frac{\ln R^2-\ln |x|}{\ln R}\right)\right],\end{equation} recalling that $h\geq \lambda_G$ and following \cite[Cor. 2.6]{FSV}, we then obtain  
\[		
\int_{\{\nabla u\neq 0\}\cap \overline{B}_R}\Big[\lambda_1|\nabla_T|\nabla u||^2+a(|\nabla u|)|\nabla u|^2 \sum_{j=1}^{n-1} k_j^2\Big]\varphi^2\ \ud\mu\leq    C^{'}  \int_{B_{R^2}\setminus B_R} \frac{a(|\nabla u|(x))}{|x|^2  \ (\ln R)^2}|\nabla u|^2 e^{G(x)}  \ud x
\]
for some constant $C^{'} >0$. 
When $R\to +\infty$, since  $a(|\nabla u|)|\nabla u|^2 e^{G(x)}$ is bounded, the r.h.s. term of the previous inequality goes to zero, and we conclude again that $u$ is one-dimensional.  

\smallskip

Assume now that $u$ is not constant.  If we take in \eqref{senza} the same test functions as above, we  get   
\[ 
\int_{\R^n}a(|\nabla u|)\left\langle(h(x) {\rm I}_n-\nabla^2 G(x))\nabla u, \nabla u\right\rangle \ud \mu(x) =0\,.
\] 
Using the fact that $u(x)=u_0(\left\langle\omega, x\rangle\right)$ and $a(t)>0$
we  obtain that  
$\left\langle(h(x) {\rm I}_n-\nabla^2 G(x))\omega, \omega\right\rangle  =0$ 
for all $x$ such that $u'_0(\left\langle\omega, x\rangle\right)\neq 0$. Since $u$ is not constant and is a solution to the elliptic equation \eqref{eqn}, the set of points such that $u'_0(\left\langle\omega, x\rangle\right)=0$ has zero measure, so, by the regularity of $G$ we conclude that 
$$
\left\langle(h(x) {\rm I}_n-\nabla^2 G(x))\omega, \omega\right\rangle  =0 \qquad\forall \ x\in\bbR^n\,,
$$
which gives \eqref{ugu1} and \eqref{ugu22}. 
\qed

\smallskip

As pointed out in \cite{CNP}, a Liouville type result follows from Theorem \ref{1Dgen}.

\begin{cor}\label{corcor}
Let $G,h,u$ satisfy the assumptions in Theorem \ref{1Dgen}. Assume further that 
$h\in C^0(\bbR^n)$ and $h(x)>\lambda_G(x)$ for some $x\in \bbR^n$. Then $u$ is constant.

In particular, if $u$ is a stable solution, that is $h\equiv 0$, and $\lambda_G(x)<0$ for some $x\in\bbR^n$, then $u$ is constant.
\end{cor} 

In the following lemma we give a sufficient condition 
for a solution $u$ to satisfy condition (a) in Theorem \ref{1Dgen}.

\begin{lem}\label{lemAlb}
Let $u$ be a weak solution to \eqref{eqn}. 
Then, for each $\varphi\in C^1_c(\bbR^n)$,
\begin{align}\label{stimaAlb1}
	\int_{\bbR^n}a(|\nabla u|)|\nabla u|^2\varphi\ud\mu=-\int_{\bbR^n}a(|\nabla u|)\left\langle \nabla u,\nabla\varphi\right\rangle u\ud\mu+\int_{\bbR^n}f(u)u\varphi\ud\mu.
\end{align}
In particular, if $t\to ta(t)\in L^{\infty}((0,+\infty))$, $u\in L^{\infty}(\bbR^n)$ and $\mu(\bbR^n)<+\infty$ then there exists $C>0$ such that
\begin{align}\label{stimaAlb2}
	\int_{\bbR^n}a(|\nabla u|)|\nabla u|^2\ud\mu\leq C\,.
\end{align}
\end{lem}
\begin{proof}
Clearly \eqref{stimaAlb1} follows by taking $u\varphi$ as test function in \eqref{weak}. 

Let us show \eqref{stimaAlb2}.
For every $R>1$ let $\Phi_R\in C^{\infty}(\bbR)$ be such that $\Phi_R(t)=1$ if $t\leq R$, $\Phi_R(t)=0$ if $t\geq R+1$ and $\Phi_R'(t)\leq 3$ for $t\in [R,R+1]$, 
and define $\varphi(x):=\Phi_R(|x|)$. 
Then $|\nabla \varphi(x)|\leq |\Phi_R'(|x|)|\leq 3$, and \eqref{stimaAlb1} yields 
\begin{align*}
\int_{B_R}a(|\nabla u|)|\nabla u|^2\ud\mu&\leq 3\int_{B_{R+1}\setminus B_R}a(|\nabla u|)|\nabla u||u|\ud\mu+\int_{B_{R+1}}|f(u)||u|\ud\mu\le C,
\end{align*}
which gives \eqref{stimaAlb2} by letting $R\to +\infty$. 

\end{proof}

\medskip

In the rest of the section we fix $G(x)=-|x|^2/2$. 
We start with a result which follows directly from Lemma \ref{stab}.

\begin{lem}\label{stabGauss}
Let $G(x):=-|x|^2/2$ and assume that $u$ is a monotone weak solution to \eqref{eqn}, i.e. there exists $i\in \{1,\ldots, n\}$ such that
\begin{align}
\partial_i u (x) >0 \quad \forall x\in\bbR^n,
\end{align}
then $ u \in C^{2}(\bbR^n)$ and $u$ is $(-1)-$stable.
\end{lem}  

\smallskip

\noindent{\it Proof of Theorem \ref{propex}.}
We start observing that $u$ is $(-1)-$stable by Lemma \ref{stab}. 

Since $\nabla^2 G(x)=-Id$ we have 
\begin{align}\label{pu}
	-1=h(x)=\lambda_G(x)=-1.
\end{align}
If $a(t)=t^{p-2}$ for some $p >1$ then
\begin{align}
	\lambda_1(t)=(p-1)t^{p-2}=(p-1)a(t)\quad \forall t>0
\end{align}
and the conclusion follows by Theorem \ref{1Dgen}. 

If $a(t)=(1+t^q)^{-\frac{1}{q}}$ with $q>1$ then
\begin{align}\label{po}
	&\lambda_1(t)=(1+t^q)^{-\frac{1}{q}}-(1+t^q)^{-\frac{q+1}{q}}t^q\leq a(t)\quad \forall t>0,\\
	\label{po2}
	&ta(t)\leq 1\quad \forall t>0.
\end{align}
By Lemma \ref{lemAlb} and \eqref{po2} there exists $C>0$ such that
\begin{align}\label{stt}
	\int_{\bbR^n}a(|\nabla u|)|\nabla u|^2\ud\mu\leq C.
\end{align}
Notice that, if $a(t)=1$ for every $t>0$, 
by Theorem \cite[Theorem 4.1]{Lun} we have $u\in H^{2}(\bbR^n,\mu)$, so that \eqref{stt} holds in this case, too.

The conclusion follows by \eqref{po}, \eqref{stt} and Theorem \ref{1Dgen}. 
\qed

\section{Solutions with Morse index bounded by the euclidean dimension}

In this section we will focus on the Ornstein-Uhlenbeck operator. 
More precisely, we will consider weak solutions $u\in H^1(\bbR^n,\mu)\cap L^{\infty}(\bbR^n)$ to 
\begin{align}\label{OU}
	\Delta u-\left\langle x,\nabla u\right\rangle+f(u)=0
\end{align}
where $f\in C^1(\bbR)$, and 
we will prove some new symmetry results for solutions 
with Morse index $k\leq n$. We recall that, 
by Theorem \cite[Theorem 4.1]{Lun}, bounded weak solutions to \eqref{OU} satisfy  $u\in H^{2}(\bbR^n,\mu) \cap L^{\infty}(\bbR^n)$. 
 
\begin{defi}\label{eigen}
A bounded weak solution $u$ to the Ornstein-Uhlenbeck operator has 
Morse index $k\in\mathbb N$ if $k$ is the maximal dimension of a subspace $X$ of $H^1(\bbR^n,\mu)$ such that
\begin{align}\label{varsec}
	Q_u(\varphi):=\int_{\bbR^n}|\nabla\varphi|^2-f'(u)\varphi^2\ud\mu<0\quad \forall \varphi\in X\setminus\{0\}.
\end{align}
\end{defi}
\begin{oss}\label{qpl}\rm
Let $u$ be a bounded solution to \eqref{OU} and let 
$L:H^{2}(\bbR^n,\mu)\to L^{2}(\bbR^n,\mu)$
be the linear operator defined as 
\begin{equation}\label{eqL}
L(v):=-\Delta v+\left\langle \nabla v, x\right\rangle-f'(u)v.
\end{equation}
Notice that $L$ is self-adjoint in $L^2(\bbR^n,\mu)$ with compact inverse, so that 
by the Spectral Theorem \cite{kato} there exists an orthonormal basis of $L^2(\bbR^n,\mu)$ consisting of eigenvectors of $L$, and 
each eigenvalue of $L$ is real.

Then, $u$ has Morse index $k$ if and only if $L$ has exactly $k$ strictly negative eigenvalues, repeated according to their geometric multiplicity (see for instance \cite[Theorem 4.1]{Lun}).
\end{oss}

The following Proposition is proved in \cite[Lemma 3.2]{CNV}.

\begin{prop}\label{autof}
Let 
$u$ be a bounded weak solution to \eqref{OU}. If for some $i=1,\ldots, n$, $u_i$ is not identically zero then it is an eigenfunction of $L$ with eigenvalue $-1$, i.e.
\begin{align}\label{eqngen2}
\int_{\bbR^n} \left\langle \nabla u_i, \nabla\varphi \right\rangle+u_i\varphi-f'(u)u_i\varphi\  \ud\mu =0, \quad \forall \varphi\in H^1(\bbR^n,\mu).
\end{align}
\end{prop}


\smallskip

We are now in a position to prove Theorem \ref{teomorse}.

\smallskip

\noindent{\it Proof of Theorem \ref{teomorse}.}
By \cite[Theorem 4.1]{Lun} every bounded weak solution to \eqref{OU} belongs to $H^{2}(\bbR^n,\mu)$, hence $u_i\in H^1(\bbR^n,\mu)$ for all $i=1,\ldots, n$. Therefore, using \eqref{eqngen2} with $u_i$ as test function we obtain 
\begin{align}
Q_u(u_i)=\int_{\bbR^n} |\nabla u_i|^2-f'(u)u^2_i\ud\mu =-\int_{\bbR^n}u^2_i\leq 0, \quad \forall i=1,\ldots, n.
\end{align}
In particular
\begin{align}
	Q_u(u_i)<0
\end{align}
for every $i=1,\ldots, n$ such that $u_i$ is not identically zero.

Let $L$ be the operator defined in \eqref{eqL}.
If $k=0$ then $u$ is stable, hence it is constant by Corollary \ref{corcor}.

If $k=1$ then, by Remark \ref{qpl} and Proposition \ref{autof}, it follows that $-1$ is the smallest eigenvalue of $L$, that is 
\begin{align}\label{wpm}
	\inf_{\varphi\in H^{1}(\bbR^n,\mu), ||\varphi||_{L^2(\bbR^n,\mu)}=1}\Big(\int_{\bbR^n}|\nabla\varphi|^2-f'(u)\varphi^2\ \ud\mu \Big)=-1.
\end{align}
Using \eqref{wpm} it follows that $u$ is $(-1)-$stable and therefore, by Theorem \ref{1Dgen}, $u$ is one-dimensional. 

Assume now $2\leq k\leq n$ and define $S:=\{i\in \{1,\ldots, n\}\ |\ u_i(x)\neq 0,\ \mbox{for some}\  x\in\bbR^n\}$ and $X:=\spa_{i\in S} \{u_i\}\subset H^1(\bbR^n,\mu)$. Clearly,
\begin{align}
	Q_u(v)<0\quad \forall v\in X\setminus\{0\}
\end{align}
therefore, by Definition \ref{eigen}, $X$ has dimension less or equal than $k$, i.e. there exists $I\subset S$ with $|I|\geq |S|-k$ such that $\{u_i\}_{i\in I}$ are linearly dependent \cite{kato}. This means that, up to an orthogonal change of variables, $u$ depends on at most $k$ variables.

Let us assume by contradiction that $u$ 
is a function of exactly $k$ variables.
We claim that $-1$ is the smallest eigenvalue of $L$, as before. Indeed, if this is not the case, then there exist $\lambda<-1$ and $v\in H^1(\bbR^n,\mu)$, with $v\not\equiv 0$, such that $L(v)=\lambda v$, therefore, by the linear independence of eigenvectors associated to different eigenvalues, it follows that $Y:=\spa\{u_i,v\}$ has dimension equal to $k+1$ and $Q_u(w)<0$ for every $w\in Y\setminus\{0\}$ which is in contradiction with the fact that $u$ has Morse index $k$. This proves that $u$ is a function of at most $(k-1)$ variables, as claimed.
\qed

\end{document}